\documentclass[11pt]{amsart}
\usepackage{mathrsfs}
\usepackage{amsmath}
\usepackage{amsthm}
\usepackage{amsfonts}
\usepackage{bbold}
\usepackage{hyperref}
\input xy
\xyoption{all}
\everymath{\displaystyle}

\title{Scaled-Free Objects}

\author{Will Grilliette}
\address{Division of Mathematics\\
Alfred University\\
109B Myers Hall\\
Alfred, NY 14802\\}
\email{w.b.grilliette@gmail.com}



\newtheorem{thm}{Theorem}[subsection]
\newtheorem{prop}[thm]{Proposition}
\newtheorem{cor}[thm]{Corollary}

\newtheorem{prop2}{Proposition}[section]

\theoremstyle{definition}
\newtheorem*{defn}{Definition}

\theoremstyle{remark}
\newtheorem{ex}[thm]{Example}

\DeclareMathOperator{\crh}{crh}
\DeclareMathOperator{\Ran}{ran}

\DeclareMathOperator{\Span}{span}
\DeclareMathOperator{\BanSpC}{BanSp_{1}}
\DeclareMathOperator{\BanSpI}{BanSp_{\infty}}
\DeclareMathOperator{\BanAlg}{BanAlg}
\newcommand{\B}{\mathcal{B}}

\newcommand{\cat}[1]{\mathscr{#1}}
\newcommand{\alg}[1]{\mathcal{#1}}

\newcommand{\Set}{\mathbf{Set}}

\newcommand{\CSetC}{\mathbf{CSet}_{1}}
\newcommand{\CSetI}{\mathbf{CSet}_{\infty}}

\begin{document}

\begin{abstract}
In this work, I address a primary issue with adapting categorical and algebraic concepts to functional analytic settings, the lack of free objects.  Using a ``normed set'' and associated categories, I describe constructions of normed objects, which build from a set to a vector space to an algebra, and thus parallel the natural progression found in algebraic settings.  Each of these is characterized as a left adjoint functor to a natural forgetful functor.  Further, the universal property in each case yields a ``scaled-free'' mapping property, which extends previous notions of ``free'' normed objects.

In subsequent papers, this scaled-free property, coupled with the associated functorial results, will give rise to a presentation theory for Banach algebras and other such objects, which inherits many properties and constructions from its algebraic counterpart.
\end{abstract}

\maketitle
\tableofcontents

\section{Introduction}

The circle of ideas regarding free objects, particularly the notion of a pair of adjoint functors, is well-known in the literature of category theory and abstract algebra, such as resources \cite{borceux1} and \cite{lang}.  However, as is well-known, free objects rarely exist in categories of normed objects over $\mathbb{F}\in\{\mathbb{R},\mathbb{C}\}$.

To summarize in the terminology and notation of this paper, let $\cat{C}$ be a subcategory of normed $\mathbb{F}$-vector spaces with $\mathbb{F}$-linear contractions.  There is a natural forgetful functor $F_{\cat{C}}:\cat{C}\to\Set$, which strips all algebraic and topological data, leaving only the underlying sets and set maps.  The free mapping property can be stated as a \emph{reflection} along the functor $F_{\cat{C}}$.  Explicitly, given a set $S$, a reflection of $S$ along $F_{\cat{C}}$ is an object $V$ of $\cat{C}$ and a function $\eta:S\to F_{\cat{C}}V$ such that given any other object $W$ of $\cat{C}$ and a function $\phi:S\to F_{\cat{C}}W$, there is a unique $\cat{C}$-map $\hat{\phi}:V\to W$ such that $F_{\cat{C}}\hat{\phi}\circ\eta=\phi$.

\begin{prop2}[Folklore]\label{fail}
Fix a nonempty set $S$.  If $\cat{C}$ contains an object not isomorphic to the zero space, then $S$ has no reflection along $F_{\cat{C}}$.
\end{prop2}

Since the free mapping property is a cornerstone to many constructions in pure algebra, particularly presentation theory, this is a most discouraging fact.  The necessity of making some sacrifice has spawned several avenues of research into generators and relations, such as \cite{blackadar1985,gerbracht,goodearl,loring2009,nassopoulos2008,cat-top,pelletier}.

The present work develops the same category of normed sets with contractive maps from \cite[p.\ 19]{gerbracht}, but also generalizes to bounded maps and identifies the properties of both.  Using these categories, the present work builds Banach spaces and Banach algebras with the analogous universal property.  Further, the constructions generalize the work of \cite{goodearl,cat-top,pelletier}.

And, it is this ``scaled-free'' mapping property that is of interest.  In subsequent papers, this scaled-free property, coupled with the associated functorial results, will give rise to a presentation theory for Banach algebras and other such objects, which inherits many properties and constructions from its algebraic counterpart.

The author would like to thank the referees of this paper for their comments and patience in its revision.  He would also like to thank Prof.\ David Pitts for his advice and help in developing these ideas.

\section{Normed Sets Revisited}\label{chap2}

This section defines an alternative working environment apart from the category of sets and explores the basic principles governing it.  This category was previously introduced in \cite[p.\ 19]{gerbracht} in the context of constructing C*-algebras and \cite[p.\ 7]{grandis2004} for combinatorial homology.

\subsection{Definitions \& Basic Results}

The objective is to construct a category so that a forgetful functor from a category of normed objects and its homomorphisms will have a left adjoint.  Explicitly, the objects will be a set with a ``sizing'' function.

\begin{defn}
A \emph{normed set} is a pair $(S,f)$, where $S$ is a set and $f$ a function from $S$ to $[0,\infty)$.  The function $f$ is called the \emph{norm function}.
\end{defn}

This structure is not new, previously considered in \cite[p.\ 19]{gerbracht} and \cite[p.\ 7]{grandis2004}, though the latter allows the use of $\infty$ as a norm-value.  This has significant impact on the structure of the associated category, as will be explained in Section \ref{cset1}.  As the author was initially unaware of the previous two references, he used the term ``crutched set'' and ``crutch function'' in his original dissertation \cite[p.\ 14]{grilliette0}.

\begin{defn}
Given two normed sets $(S,f)$ and $(T,g)$, a function $\phi:S\to T$ is \emph{bounded} if there is $M\geq 0$ such that for all $s\in S$, $g\left(\phi(s)\right)\leq M f(s)$.  Let
\[
\crh(\phi):=\inf\left\{M\in[0,\infty):g\left(\phi(s)\right)\leq M f(s)\forall s\in S\right\},
\]
the \emph{bound constant} of $\phi$.  If $\crh(\phi)\leq 1$, $\phi$ is \emph{contractive}.
\end{defn}

The contractive notion was also visited in \cite[p.\ 19]{gerbracht} and \cite[p.\ 7]{grandis2004}.  The author used the terms ``crutch bound'' and ``constrictive'' in his original dissertation \cite[p.\ 15]{grilliette0}.

With these notions, adaptations of the standard functional analysis proofs can be used to prove the following foundational results.

\begin{prop}[Boundedness Criterion]\label{boundedness}
Let $(S,f)$ and $(T,g)$ be normed sets.  A function $\phi:S\to T$ is bounded if and only if
\[
\sup\left(\left\{\frac{g(\phi(s))}{f(s)}:s\not\in f^{-1}(0)\right\}\cup\{0\}\right)<\infty
\]
and $g(\phi(s))=0$ for all $s\in f^{-1}(0)$.  In this case, $\crh(\phi)$ equals the above supremum and
\[
g(\phi(s))\leq \crh(\phi)f(s).
\]
for all $s\in S$.
\end{prop}

\begin{cor}[Composition]\label{comp-bounded}
Let $(S,f)$, $(T,g)$, and $(U,h)$ be normed sets and $\phi:(S,f)\to(T,g)$ and $\psi:(T,g)\to(U,h)$ be bounded.  Then, $\psi\circ\phi:S\to U$ is bounded and
\[
\crh(\psi\circ\phi)\leq\crh(\psi)\crh(\phi).
\]
If $\phi$ and $\psi$ are contractive, so is $\psi\circ\phi$.
\end{cor}

\subsection{Category of Normed Sets \& Contractive Maps}\label{cset1}

Next, a detailed study is conducted of normed sets and contractive functions between them.  This combination of objects and maps was considered previously in \cite{gerbracht}.  For notation, let $\CSetC$ denote the category of normed sets with contractive maps.

With this new structure defined, one considers some of its basic properties and constructions.  Most of these are identical the $\Set$ case, though most interestingly, the norm function in each case immediately resembles its counterpart in normed structures.

The following proposition characterizes the standard types of morphisms for $\CSetC$.  Isomorphisms were mentioned briefly in \cite[p.\ 7]{grandis2004}.  Also, this proposition adds precision to \cite[Remark 1.1.9]{gerbracht} in regard to sections and retractions.  The proofs are nearly identical to the set-theoretic versions and will be omitted.

\begin{prop}\label{morphisms}
Let $(S,f)$ and $(T,g)$ be normed sets and $\phi:(S,f)\to(T,g)$ be contractive.  The following characterizations hold.
\begin{enumerate}
\item $\phi$ is a monomorphism in $\CSetC$ iff $\phi$ is one-to-one;
\item $\phi$ is an epimorphism in $\CSetC$ iff $\phi$ is onto;
\item $\phi$ is a section in $\CSetC$ iff $\phi$ is one-to-one, $g\circ\phi=f$, and for all $t\not\in\phi(S)$, there is $s_{t}\in S$ such that $f\left(s_{t}\right)\leq g(t)$;
\item $\phi$ is a retraction in $\CSetC$ iff for all $t\in T$, there is $s_{t}\in S$ such that $\phi\left(s_{t}\right)=t$ and $f\left(s_{t}\right)=g(t)$;
\item\label{iso} $\phi$ is an isomorphism in $\CSetC$ iff $\phi$ is one-to-one, onto, and $g\circ\phi=f$.
\end{enumerate}
\end{prop}

The construction of the equalizer and coequalizer of parallel maps are the same as mentioned in \cite[p.\ 7]{grandis2004}, characterizing substructures and quotients, respectively.  Likewise, a small coproduct is identical to the characterization in \cite{grandis2004}, and it gives a standard decomposition of any normed set as a coproduct of singletons.

However, unlike \cite{grandis2004}, exclusion of $\infty$ as a norm-value changes the product structure.

\begin{prop}[Products]
For an index set $I$, let $\left(S_{i},f_{i}\right)$ be normed sets for $i\in I$.  Define
\[
P:=\left\{\vec{s}\in\Set\left(I,\bigcup_{i\in I}S_{i}\right):\vec{s}(i)\in S_{i}\forall i\in I,\sup\left\{f_{i}\left(\vec{s}(i)\right):i\in I\right\}<\infty\right\},
\]
$f:P\to[0,\infty)$ by $f\left(\vec{s}\right):=\sup\left\{f_{i}\left(\vec{s}(i)\right):i\in I\right\}$, and $\pi_{i}:P\to S_{i}$ by $\pi_{i}\left(\vec{s}\right):=\vec{s}(i)$.  Then, $(P,f)$ equipped with $\left(\pi_{i}\right)_{i\in I}$ is a product of $\left(\left(S_{i},f_{i}\right)\right)_{i\in I}$ in $\CSetC$.
\end{prop}

Since $\infty$ is not allowed, any sequence of elements whose norm-values would be unbounded must be excluded.  The proof is akin to the product characterization for the category of Banach spaces with contractive maps.

This difference in the construction of the product forebodes a difference between $\CSetC$ and the category of \cite{grandis2004}.  Indeed, these two categories are not equivalent, which can be shown by counting their projective objects.  From \cite{grandis2004}, let $\mathbf{NSet}$ denote the category of normed sets and contractive maps, which allow $\infty$.  Like $\Set$, any set equipped with the constant-$\infty$ norm is projective with respect to all epimorphisms.  However, disallowing $\infty$ in $\CSetC$ almost completely forbids this behavior.

\begin{prop}\label{proj-inj}
Let $(S,f)$ be a normed set.
\begin{enumerate}
\item $(S,f)$ is projective relative to all epimorphisms in $\CSetC$ iff $S=\emptyset$.
\item $(S,f)$ is injective relative to all monomorphisms in $\CSetC$ iff $S\neq\emptyset$ and $f=0$.
\end{enumerate}
\end{prop}

\begin{proof}
\begin{enumerate}

\item $(\Leftarrow)$ The empty set, equipped with the empty function to $[0,\infty)$, is initial in $\CSetC$.  Hence, it is trivially projective with respect to any class of maps.

$(\neg\Leftarrow\neg)$ For purposes of contradiction, assume that $S\neq\emptyset$ and $(S,f)$ is projective relative to all epimorphisms.  For each $n\in\mathbb{N}$, define $g,h_{n}:S\to [0,\infty)$ by $g(s):=0$ and $h_{n}(s):=n$.  Also, let $\phi,\alpha_{n}:S\to S$ by $\phi(s):=\alpha_{n}(s):=s$.  Then, consider the following diagram in $\CSetC$ for each $n$.
\[
\xymatrix{
&	\left(S,h_{n}\right)\ar@{->>}[d]^{\alpha_{n}}\\
(S,f)\ar[r]_{\phi}	&	(S,g)\\
}
\]
Since $\alpha_{n}$ is onto and $(S,f)$ projective to epimorphisms, there must be a contractive $\phi_{n}:(S,f)\to(S,h_{n})$ such that $\phi=\alpha_{n}\circ\phi_{n}$.  Then, for each $s\in S$ and $n\in\mathbb{N}$,
\[
s=\phi(s)=(\alpha\circ\phi_{n})(s)=\phi_{n}(s)
\]
and
\[
n=(h_{n}\circ\phi_{n})(s)\leq f(s).
\]
Thus, $f$ cannot have a finite value, contradicting that $(S,f)$ was in $\CSetC$.

\item $(\Rightarrow)$ Assume that $(S,f)$ is injective relative to all monomorphisms.  Let $\mathbf{0}_{S}:\emptyset\to S$ and $\mathbf{0}_{\{0\}}:\emptyset\to\{0\}$ be the empty functions into $S$ and $\{0\}$, respectively.  Consider the following diagram in $\CSetC$.
\[
\xymatrix{
(S,f)\ar@{<-}[d]_{\mathbf{0}_{S}}\\
(\emptyset,\mathbf{0}_{[0,\infty)})\textrm{  }\ar@{>->}[r]_{\mathbf{0}_{\{0\}}}	&	\{(0,0)\}.
}
\]
As $(S,f)$ is injective relative to $\mathbf{0}_{\{0\}}$, there must be a contractive map from $\{(0,0)\}$ to $(S,f)$.  Hence, there is a function from a nonempty set into $S$, forcing $S\neq\emptyset$.

Define $h:S\to[0,\infty)$ by $h(s):=0$.  Also, let $\phi,\alpha:S\to S$ by $\phi(s):=\alpha(s):=s$.  Then, consider the following diagram in $\CSetC$.
\[
\xymatrix{
(S,f)\ar@{<-}[d]_{\phi}\\
(S,f)\textrm{  }\ar@{>->}[r]_{\alpha}	&	(S,h).
}
\]
Then, there is a contraction $\hat{\phi}:(S,h)\to(S,f)$ such that $\phi=\hat{\phi}\circ\alpha$.  Then, for each $s\in S$,
\[
s
=\phi(s)
=\left(\hat{\phi}\circ\alpha\right)(s)
=\hat{\phi}(s)
\]
and
\[
0
\leq f(s)
=\left(f\circ\hat{\phi}\right)(s)
\leq h(s)
=0.
\]

$(\Leftarrow)$ Assume that $f=0$ and $S\neq\emptyset$.  Let $(T,g)$ and $(U,h)$ be normed sets and $\alpha:(T,g)\to(U,h)$ be a monomorphism.  Define $\hat{U}:=\Ran(\alpha)$ and observe that $\alpha|^{\hat{U}}$ is bijective.  Given any $\phi:T\to S$, choose any $s_{0}\in S$ and define $\hat{\phi}:U\to S$ by
\[
\hat{\phi}(u):=\left\{\begin{array}{cc}
\phi(s),	&	u=\alpha(s),\\
s_{0},	&	u\not\in\hat{U}.\\
\end{array}\right.
\]
As $\alpha$ is one-to-one, this is a well-defined function.  By design, $\phi=\hat{\phi}\circ\alpha$, and since $f=0$, $\hat{\phi}$ is trivially contractive.

\end{enumerate}
\end{proof}

There is precisely one isomorphism class of a projective object relative to all epimorphisms in $\CSetC$, but $\Set$ and $\mathbf{NSet}$ both have a proper class of such isomorphism classes.  Hence, the distinction follows.

\begin{cor}
$\CSetC$ is equivalent to neither $\Set$ nor $\mathbf{NSet}$ as categories.
\end{cor}

\subsection{Category of Normed Sets \& Bounded Maps}\label{cseti}

Likewise, normed sets and bounded functions between them, denoted as $\CSetI$, can be studied, comparing this structure to $\CSetC$.  At first glance, $\CSetI$ is very similar to $\CSetC$, and most of its constructions are identical.  However, there are some notable distinctions between the two, reminiscent of the differences between considering Banach spaces with bounded linear maps and contractive linear maps.

The standard types of morphisms are characterized much like in $\CSetC$, but also similar to Banach spaces and bounded maps.  In particular, the notion of a section corresponds to the idea of ``bounded below''.

\begin{prop}
Let $(S,f)$ and $(T,g)$ be normed sets and $\phi:(S,f)\to(T,g)$ be bounded.  Define $K:=T\setminus\phi(S)$, $h:=g|_{K}$, and
\[
\lambda:=\inf\left\{\frac{g(\phi(s))}{f(s)}:s\not\in f^{-1}(0)\right\}.
\]
\begin{enumerate}
\item $\phi$ is a monomorphism in $\CSetI$ iff $\phi$ is one-to-one;
\item $\phi$ is an epimorphism in $\CSetI$ iff $\phi$ is onto;
\item $\phi$ is a section in $\CSetI$ iff $\phi$ is one-to-one, $\lambda>0$, and there is a bounded function $\alpha:(K,h)\to(S,f)$;
\item $\phi$ is a retraction in $\CSetI$ iff there are $\left(s_{t}\right)_{t\in T}\subseteq S$ such that $\phi\left(s_{t}\right)=t$ for all $t\in T$, $f\left(s_{t}\right)=0$ for all $t\in g^{-1}(0)$ and
\[
\sup\left(\left\{\frac{f\left(s_{t}\right)}{g(t)}:t\not\in g^{-1}(0)\right\}\cup\{0\}\right)<\infty;
\]
\item\label{iso2} $\phi$ is an isomorphism in $\CSetI$ iff $\phi$ is one-to-one, onto, and $\lambda>0$.
\end{enumerate}
\end{prop}

The equalizer and coequalizer of parallel maps is constructed as in $\CSetC$, as are the product and coproduct for a finite index set.  However, $\CSetI$ is neither complete nor cocomplete as a category.  The examples are similar the case of Banach spaces and bounded maps, where the supposed universal map would be forced to be unbounded.

\begin{ex}
For $n\in\mathbb{N}$, let $S_{n}:=[0,\infty)$ and $f_{n}:S_{n}\to[0,\infty)$ by $f_{n}(\lambda):=\lambda$.  The family $\left(S_{n},f_{n}\right)_{n\in\mathbb{N}}$ does not have product in $\CSetI$.
\end{ex}

\begin{ex}
For $n\in\mathbb{N}$, define $S_{n}:=\{0\}$ and $f_{n}:S_{n}\to[0,\infty)$ by $f_{n}(0):=1$.  The family $\left(S_{n},f_{n}\right)_{n\in\mathbb{N}}$ does not have a coproduct in $\CSetI$.
\end{ex}

As such, $\CSetI$ is a distinct category from the ones previously mentioned.

\begin{cor}
$\CSetI$ is equivalent as categories to neither $\Set$, $\CSetC$, nor $\mathbf{NSet}$.
\end{cor}

\section{Scaled-Free Constructions}\label{construction}

This section concerns the construction of building normed algebraic objects, specifically Banach spaces and algebras, from normed sets.  The main idea in each case is to build the appropriate free algebraic object on the set and then use the set's norm function to build the corresponding algebraic norm, generalizing the constructions of \cite{blackadar1985}, \cite{gerbracht}, and \cite{goodearl} with the viewpoint of \cite{cat-top}.  The use of the norm function is analogous to the ``$\mathcal{X}$-norms'' in \cite{hadwin2003}, but the universal objects created here are normed structures, as opposed to a general topological ones.

\subsection{Banach Spaces}

Let $\mathbb{F}\mathbf{Ban}_{\infty}$ denote the category of $\mathbb{F}$-Banach spaces with bounded $\mathbb{F}$-linear maps.  The forgetful functor $F_{\mathbb{F}\mathbf{Ban}_{\infty}}^{\CSetI}:\mathbb{F}\mathbf{Ban}_{\infty}\to\CSetI$ drops all of the linear structure.

Given a normed set $(S,f)$, construction of a reflection along $F_{\mathbb{F}\mathbf{Ban}_{\infty}}^{\CSetI}$ would proceed along natural lines.  One builds an $\mathbb{F}$-vector space with basis $S\setminus f^{-1}(0)$ and completes in an appropriate universal norm.  However, this construction is readily characterized as a weighted $\ell^{1}$-space.

Specifically, let $\hat{S}:=S\setminus f^{-1}(0)$, $\mu_{f}:\mathcal{P}\left(\hat{S}\right)\to[0,\infty]$ by $\mu_{f}(T):=\sum_{s\in T}f(s)$, and $\mathcal{V}_{S,f}:=\ell_{\mathbb{F}}^{1}\left(\hat{S},\mu_{f}\right)$.  Define $\zeta_{S,f}:(S,f)\to\alg{V}_{S,f}$ by
\[
\zeta_{S,f}(s):=\left\{\begin{array}{cc}
0,	&	s\in f^{-1}(0),\\
\delta_{s},	&	s\not\in f^{-1}(0),\\
\end{array}\right.
\]
where $\delta_{s}$ is the point mass at $s\in\hat{S}$.

\begin{thm}[Reflection Characterization, $\mathbb{F}\mathbf{Ban}_{\infty}$]\label{univprop}
Given an $\mathbb{F}$-Banach space $W$ and a bounded map $\phi:(S,f)\to F_{\mathbb{F}\mathbf{Ban}_{\infty}}^{\CSetI}W$, there is a unique bounded $\mathbb{F}$-linear map $\hat{\phi}:\alg{V}_{S,f}\to W$ such that $F_{\mathbb{F}\mathbf{Ban}_{\infty}}^{\CSetI}\hat{\phi}\circ\zeta_{S,f}=\phi$.  Moreover,
\[
\crh(\phi)=\left\|\hat{\phi}\right\|_{\B\left(\alg{V}_{S,f},W\right)}.
\]
\end{thm}

\begin{proof}

Define $\tilde{\phi}:\Span\left\{\delta_{s}:s\in\hat{S}\right\}\to W$ on the standard basis by $\tilde{\phi}\left(\delta_{s}\right):=\phi(s)$ for all $s\in\hat{S}$.  For any finite $E\subseteq\hat{S}$ and scalars $\left(\lambda_{s}\right)_{s\in E}$,
\[\begin{array}{rcl}
\left\|\tilde{\phi}\left(\sum_{s\in E}\lambda_{s}\delta_{s}\right)\right\|_{W}	&	\leq	&	\sum_{s\in E}\left|\lambda_{s}\right|\left\|\tilde{\phi}\left(\delta_{s}\right)\right\|_{W}\\[15pt]
&	=	&	\sum_{s\in E}\left|\lambda_{s}\right|\left\|\phi(s)\right\|_{W}\\[15pt]
&	\leq	&	\sum_{s\in E}\left|\lambda_{s}\right|\crh(\phi)f(s)\\[15pt]
&	=	&	\crh(\phi)\left\|\sum_{s\in E}\lambda_{s}\delta_{s}\right\|_{\alg{V}_{S,f}}.\\[15pt]
\end{array}\]
Thus, $\tilde{\phi}$ can be extended by continuity to $\hat{\phi}:\alg{V}_{S,f}\to W$.  By design, $F_{\mathbb{F}\mathbf{Ban}_{\infty}}^{\CSetI}\hat{\phi}\circ\zeta_{S,f}=\phi$, and uniqueness follows from the mapping of the basis vectors via $\zeta_{S,f}$.

By the norm computation above,
\[
\crh(\phi)\geq\left\|\hat{\phi}\right\|_{\B\left(\alg{V}_{S,f},W\right)},
\]
and equality is achieved using the basis vectors.

\end{proof}

Further, since $(S,f)$ was arbitrary, the following functorial result is obtained.

\begin{cor}[Left Adjoint Functor, $\mathbb{F}\mathbf{Ban}_{\infty}$]\label{adjoint}
There is a unique functor $\mathbb{F}\BanSpI:\CSetI\to\mathbb{F}\mathbf{Ban}_{\infty}$ such that $\mathbb{F}\BanSpI(S,f)=\alg{V}_{S,f}$, which is left adjoint to $F_{\mathbb{F}\mathbf{Ban}_{\infty}}^{\CSetI}$.
\end{cor}

The numeric condition in Theorem \ref{univprop} actually shows that there is a second adjoint relationship here.  Specifically, consider the category of $\mathbb{F}$-Banach spaces with contractive $\mathbb{F}$-linear maps, denoted as $\mathbb{F}\mathbf{Ban}_{1}$.  Let $F_{\mathbb{F}\mathbf{Ban}_{1}}^{\CSetC}:\mathbb{F}\mathbf{Ban}_{1}\to\CSetC$ be the restriction of $F_{\mathbb{F}\mathbf{Ban}_{\infty}}^{\CSetI}$ to $\mathbb{F}\mathbf{Ban}_{1}$.

\begin{cor}[Left Adjoint Functor, $\mathbb{F}\mathbf{Ban}_{1}$]
There is a unique functor $\mathbb{F}\BanSpC:\CSetC\to\mathbb{F}\mathbf{Ban}_{1}$ such that $\mathbb{F}\BanSpC(S,f)=\alg{V}_{S,f}$, which is left adjoint to $F_{\mathbb{F}\mathbf{Ban}_{1}}^{\CSetC}$.
\end{cor}

The universal property of Theorem \ref{univprop} seems very similar to the classical free mapping property, but depends on the boundedness of maps.  However, this can be somewhat negotiated at the cost of scalar multiplication.  The following variation of Theorem \ref{univprop} is termed the \emph{scaled-free mapping property}, which more closely mimics the classical free mapping property.

\begin{cor}[Scaled-Free Mapping Property]\label{scaledfree}
Let $(S,f)$ be a normed set and $W$ be an $\mathbb{F}$-Banach space.  For any function $\phi:S\to W$, there is a unique contractive $\mathbb{F}$-linear map $\hat{\phi}:\mathbb{F}\BanSpC(S,f)\to W$ such that for all $s\in S$,
\[
\|\phi(s)\|_{W}\cdot\left(\hat{\phi}\circ\zeta_{S,f}\right)(s)=f(s)\cdot\phi(s).
\]
\end{cor}

This adjoint characterization should be compared to the well-known unit ball functor.  Explicitly, the functor $U_{\mathbb{F}}:\mathbb{F}\mathbf{Ban}_{1}\to\Set$ by associating a Banach space with its closed unit ball and a contraction with its restriction to the unit ball.  As shown in \cite{joyofcats}, every set $S$ has a reflection along this functor, namely $\ell^{1}_{\mathbb{F}}(S)$.

However, with the functor $U_{\mathbb{F}}$, the norm has been hardcoded by the choice of the unit ball.  That is, any element of $S$ \emph{must} be sent to an element of norm at most 1.

The characterization presented here has allowed the norms of generators to vary, preserving the numeric data in a function rather than the choice of a subset.  Indeed, the $f$ in Theorem \ref{univprop} is fixed prior to construction, but has no restriction otherwise.  In particular, it need not be constant or bounded.

Also, the functors $F_{\mathbb{F}\mathbf{Ban}_{\infty}}^{\CSetI}$ and $F_{\mathbb{F}\mathbf{Ban}_{1}}^{\CSetC}$ only remove structure, not altering the underlying set in any way.  This aspect seems to give a more natural ``forgetful'' feel like the classical situation of algebraic free objects.

Theorem \ref{univprop} states that the properties of the unit ball functor are recovered and extended to the case of bounded $\mathbb{F}$-linear maps.  Arguably, one can choose to scale all generators to norm 1, but in some cases, it may be preferable to let individual generators have different norm values.

\subsection{Banach Algebras}

Let $\mathbb{F}\mathbf{BanAlg}_{\infty}$ denote the category of $\mathbb{F}$-Banach algebras with bounded $\mathbb{F}$-algebra homomorphisms.  The forgetful functor $F_{\mathbb{F}\mathbf{BanAlg}_{\infty}}^{\CSetI}:\mathbb{F}\mathbf{BanAlg}_{\infty}\to\CSetI$ drops all of the algebraic structure.

As in the previous section, one would like to build a reflection along this forgetful functor for any given normed set.  However, due to the introduction of multiplication, this is not possible except in trivial cases.

\begin{prop}\label{fail4}
A normed set $(S,f)$ has a reflection along $F_{\mathbb{F}\mathbf{BanAlg}_{\infty}}^{\CSetI}$ if and only if $S=f^{-1}(0)$.  In this case, the reflection is the zero algebra equipped with the constant map from $(S,f)$.
\end{prop}

\begin{proof}

$(\Leftarrow)$ This case is an exercise.

$(\neg\Leftarrow\neg)$ Assume that $S\neq f^{-1}(0)$.  For purposes of contradiction, assume that there is an $\mathbb{F}$-Banach algebra $R$ equipped with a bounded map $\eta:(S,f)\to F_{\mathbb{F}\mathbf{BanAlg}_{\infty}}^{\CSetI}R$ such that for any $\mathbb{F}$-Banach algebra $\alg{A}$ and bounded map $\phi:(S,f)\to F_{\mathbb{F}\mathbf{BanAlg}_{\infty}}^{\CSetI}\alg{A}$, there is a unique bounded $\mathbb{F}$-algebra homomorphism $\hat{\phi}:R\to\alg{A}$ satisfying $F_{\mathbb{F}\mathbf{BanAlg}_{\infty}}^{\CSetI}\hat{\phi}\circ\eta=\phi$.  Let $r_{s}:=\eta(s)$ for all $s\in S$.

Define $\phi:S\to\mathbb{F}$ by $\phi(s):=f(s)$, the norm function itself.  Then, $\crh(\phi)=1$ so there is a unique bounded $\mathbb{F}$-algebra homomorphism $\hat{\phi}:R\to\mathbb{F}$ such that $\hat{\phi}\circ\eta=\phi$.  For all $s\in S$,
\[
f(s)=|\phi(s)|
\leq\left\|\hat{\phi}\right\|_{\B(R,\mathbb{F})}\left\|r_{s}\right\|_{R}
\leq\left\|\hat{\phi}\right\|_{\B(R,\mathbb{F})}\crh(\eta)f(s)
\]
since $\hat{\phi}$ and $\eta$ are bounded.  For $s\not\in f^{-1}(0)$, a division yields
\[
1\leq\left\|\hat{\phi}\right\|_{\B(R,\mathbb{F})}\crh(\eta),
\]
forcing $\crh(\eta)\neq 0$.

Define $\psi:S\to\mathbb{F}$ by $\psi(s):=2\crh(\eta)f(s)$.  Notice that $\crh(\psi)=2\crh(\eta)$ so there is a unique bounded $\mathbb{F}$-algebra homomorphism $\hat{\psi}:R\to\mathbb{F}$ such that $\hat{\psi}\circ\eta=\psi$.  For $n\in\mathbb{N}$,
\[
\hat{\psi}\left(r_{s}^{n}\right)=\hat{\psi}\left(r_{s}\right)^{n}
=\psi\left(s\right)^{n}
=2^{n}\crh(\eta)^{n}f(s)^{n}
\]
and
\[
\left|\hat{\psi}\left(r_{s}^{n}\right)\right|\leq\left\|\hat{\psi}\right\|_{\B(R,\mathbb{F})}\left\|r_{s}^{n}\right\|_{R}
\leq\left\|\hat{\psi}\right\|_{\B(R,\mathbb{F})}\left\|r_{s}\right\|_{R}^{n}
\leq\left\|\hat{\psi}\right\|_{\B(R,\mathbb{F})}\crh(\eta)^{n}f(s)^{n}.
\]
Combining these for $s\not\in f^{-1}(0)$, a division yields
\[
2^{n}\leq\left\|\hat{\psi}\right\|_{\B(R,\mathbb{F})},
\]
contradicting that $\hat{\psi}$ was bounded.

\end{proof}

This is initially discouraging like Proposition \ref{fail}.  However, observe that the cause of the failure here was the ability to send a generator to a value potentially larger than its norm value.  This, coupled with the multiplicative structure, forced the norm of the fictional universal map to grow without bound.

This behavior is disallowed in the contractive case, where the scaled-free construction works perfectly well.  To see this, consider the category of $\mathbb{F}$-Banach algebras with contractive $\mathbb{F}$-algebra homomorphisms, denoted as $\mathbb{F}\mathbf{BanAlg}_{1}$.  There is a natural forgetful functor $F_{\mathbb{F}\mathbf{BanAlg}_{1}}^{\CSetC}:\mathbb{F}\mathbf{BanAlg}_{1}\to\CSetC$ by dropping all of the algebraic properties, leaving only the norm functions and the contractivity of the maps.

For a normed set $(S,f)$, construction of a reflection along $F_{\mathbb{F}\mathbf{BanAlg}_{1}}^{\CSetC}$ would proceed along natural lines.  One builds a free $\mathbb{F}$-algebra on $S\setminus f^{-1}(0)$ and completes in an appropriate universal norm.  However, this construction can be realized in another way.  From Theorem \ref{univprop}, a universal $\mathbb{F}$-Banach space can be created from $(S,f)$.  In \cite[Satz 1]{leptin1969}, a universal $\mathbb{F}$-Banach algebra can be created from this space, the Banach tensor algebra.

To make this explicit, let $F_{\mathbb{F}\mathbf{BanAlg}_{1}}^{\mathbb{F}\mathbf{Ban}_{1}}:\mathbb{F}\mathbf{BanAlg}_{1}\to\mathbb{F}\mathbf{Ban}_{1}$ be the natural forgetful functor given by dropping only the multiplicative properties.  Given an $\mathbb{F}$-Banach space $V$, the construction of \cite{leptin1969} forms the $\ell^{1}$-direct sum of projective-tensor powers of $V$,
\[
\mathscr{T}_{1}(V):=\oplus_{1}V^{\otimes n}
\]
equipped with multiplication by the canonical isomorphism $V^{\otimes n}\otimes V^{\otimes m}\to V^{\otimes (m+n)}$.  Letting $\iota_{V}:V\to\mathscr{T}_{1}(V)$ be the inclusion map into the first tensor power of $V$ in $\mathscr{T}_{1}(V)$, the following is a restatement of \cite[Satz 1]{leptin1969} in the notation of the current paper.

\begin{thm}[Satz 1, \cite{leptin1969}]
Given an $\mathbb{F}$-Banach algebra $\alg{A}$ and a contractive linear $\phi:V\to F_{\mathbb{F}\mathbf{BanAlg}_{1}}^{\mathbb{F}\mathbf{Ban}_{1}}\alg{A}$, there is a unique contractive $\mathbb{F}$-algebra homomorphism $\hat{\phi}:\mathscr{T}_{1}(V)\to\alg{A}$ such that $F_{\mathbb{F}\mathbf{BanAlg}_{1}}^{\mathbb{F}\mathbf{Ban}_{1}}\hat{\phi}\circ\iota_{V}=\phi$.
\end{thm}

This gives an associated functorial result since $V$ was arbitrary.

\begin{cor}[Left Adjoint Functor from Satz 1, \cite{leptin1969}]
There is a unique functor $\mathscr{T}_{1}:\mathbb{F}\mathbf{Ban}_{1}\to\mathbb{F}\mathbf{BanAlg}_{1}$ defined on objects as above, which is left adjoint to $F_{\mathbb{F}\mathbf{BanAlg}_{1}}^{\mathbb{F}\mathbf{Ban}_{1}}$.
\end{cor}

A quick check shows that $F_{\mathbb{F}\mathbf{BanAlg}_{1}}^{\CSetC}=F_{\mathbb{F}\mathbf{Ban}_{1}}^{\CSetC}\circ F_{\mathbb{F}\mathbf{BanAlg}_{1}}^{\mathbb{F}\mathbf{Ban}_{1}}$, so one can invoke the closure of right adjoints on composition.  Hence, its left adjoint is given by $\mathbb{F}\BanAlg:=\mathscr{T}_{1}\circ\mathbb{F}\BanSpC$.  Letting $\kappa_{S,f}:=\iota_{\alg{V}_{S,f}}\circ\zeta_{S,f}$ be the embedding of the generation set, the universal property can be stated as follows.

\begin{thm}[Reflection Characterization, $\mathbb{F}\mathbf{BanAlg}_{1}$]
Given an $\mathbb{F}$-Banach algebra $\alg{A}$ and a contractive map $\phi:(S,f)\to F_{\mathbb{F}\mathbf{BanAlg}_{1}}^{\CSetC}\alg{A}$, there is a unique contractive $\mathbb{F}$-algebra homomorphism $\hat{\phi}:\mathbb{F}\BanAlg(S,f)\to\alg{A}$ such that $F_{\mathbb{F}\mathbf{BanAlg}_{1}}^{\CSetC}\hat{\phi}\circ\kappa_{S,f}=\phi$.
\end{thm}

The characterization on objects becomes nearly immediate by use of \cite[Satz 1]{leptin1969} with the disjoint union coproduct of \cite[p.\ 7]{grandis2004} and the knowledge that left adjoints preserve coproducts.  Given a normed set $(S,f)$,
\[\begin{array}{rcl}
\mathbb{F}\BanAlg(S,f)	&	\cong_{\mathbb{F}\mathbf{BanAlg}_{1}}	&	\coprod_{s\in S}\mathbb{F}\BanAlg(s,f(s))\\[15pt]
&	\cong_{\mathbb{F}\mathbf{BanAlg}_{1}}	&	\coprod_{s\not\in f^{-1}(0)}\mathscr{T}_{1}(\mathbb{F})\\[15pt]
&	\cong_{\mathbb{F}\mathbf{BanAlg}_{1}}	&	\coprod_{s\not\in f^{-1}(0)}\ell^{1}_{\mathbb{F}}(\mathbb{N}),\\[15pt]
\end{array}\]
where $\ell^{1}_{\mathbb{F}}(\mathbb{N})$ is equipped with the convolution product.  Note that the coproduct in $\mathbb{F}\mathbf{BanAlg}_{1}$ is the free product defined in \cite[p.\ 318]{gronbaek1992}.  An analogous scaled-free mapping property holds like Corollary \ref{scaledfree}, but will be omitted for brevity.

As a consequence, the adjoint characterization of Corollary \ref{adjoint} and the failure in Proposition \ref{fail4} yield the following nonexistence result due to the closure of left adjoints on composition.

\begin{cor}
There cannot exist a functor which is left adjoint to the forgetful functor from $\mathbb{F}\mathbf{BanAlg}_{\infty}$ to $\mathbb{F}\mathbf{Ban}_{\infty}$.
\end{cor}

That is, there is no analog of the Banach tensor algebra when using bounded homomorphisms.

\subsection{Failure of Hilbert Spaces}

Consider the category of $\mathbb{F}$-Hilbert spaces and $\mathbb{F}$-linear contractions, $\mathbb{F}\mathbf{Hilb}_{1}$.    Let $F_{\mathbb{F}\mathbf{Hilb}_{1}}^{\CSetC}:\mathbb{F}\mathbf{Hilb}_{1}\to\CSetC$ be the restriction of $F_{\mathbb{F}\mathbf{Ban}_{\infty}}^{\CSetI}$ to $\mathbb{F}\mathbf{Hilb}_{1}$.  As in the previous failure cases, most interesting normed sets cannot have a reflection along this functor.

\begin{prop}\label{hilbert}
A normed set $(S,f)$ has a reflection along $F_{\mathbb{F}\mathbf{Hilb}_{1}}^{\CSetC}$ if and only if $S$ has either no element or precisely one element of nonzero norm.  In these cases, the reflections are the zero space and the field, respectively, equipped with the norm function from $(S,f)$ to $\mathbb{F}$.
\end{prop}

\begin{proof}

$(\Leftarrow)$ This case is an exercise.

$(\neg\Leftarrow\neg)$ Consider first when $\mathbb{F}=\mathbb{C}$.  For purposes of contradiction, assume that there is an $\mathbb{C}$-Hilbert space $R$ equipped with a contractive map $\eta:(S,f)\to F_{\mathbb{C}\mathbf{Hilb}_{1}}^{\CSetC}R$ such that for any $\mathbb{C}$-Hilbert space $\alg{H}$ and contractive map $\phi:(S,f)\to F_{\mathbb{C}\mathbf{Hilb}_{1}}^{\CSetC}\alg{H}$, there is a unique contractive $\mathbb{C}$-linear map $\hat{\phi}:R\to\alg{H}$ satisfying $F_{\mathbb{C}\mathbf{Hilb}_{1}}^{\CSetC}\hat{\phi}\circ\eta=\phi$.  Let $v_{s}:=\eta(s)$ for all $s\in S$.

First, the norms of each generator are determined.  Since $\eta$ is contractive, $\left\|v_{s}\right\|_{R}\leq f(s)$ for all $s\in S$.  Consider the function $\psi:S\to\mathbb{C}$ by $\phi(s):=f(s)$, the norm function itself.  Then, $\crh\left(\psi\right)=1$ so there is a unique $\mathbb{C}$-linear contraction $\hat{\psi}:R\to\mathbb{C}$ such that $\hat{\psi}\circ\eta=\psi$.  Therefore, for all $s\in S$,
\[
\left\|v_{s}\right\|_{R}\geq\left|\hat{\psi}\left(v_{s}\right)\right|
=f(s),
\]
which forces equality.

Next, consider the inner product of two generators via the polarization identity.  For $n\in\mathbb{N}$ and $s\neq t$,
\[
\left\|v_{s}+\imath^{n}v_{t}\right\|_{R}
\leq\left\|v_{s}\right\|_{R}+\left\|v_{t}\right\|_{R}
\leq f(s)+f(t).
\]
Define $\phi_{s,t,n}:S\to\mathbb{C}$ by
\[
\phi_{s,t,n}(u):=\left\{\begin{array}{cc}
f(s),	&	u=s,\\
\imath^{-n} f(t),	&	u=t,\\
0,	&	\textrm{otherwise}.
\end{array}\right.
\]
Then, $\crh\left(\phi_{s,t,n}\right)=1$ so there is a unique $\mathbb{C}$-linear contraction $\hat{\phi}_{s,t,n}:R\to\mathbb{C}$ such that $\hat{\phi}\circ\eta=\phi$.  Hence,
\[\begin{array}{rcl}
\left\|v_{s}+\imath^{n}v_{t}\right\|_{R}	&	\geq	&	\left|\hat{\phi}_{s,t,n}\left(v_{s}+\imath^{n}v_{t}\right)\right|\\[8pt]
&	=	&	\left|\hat{\phi}_{s,t,n}\left(v_{s}\right)+\imath^{n}\hat{\phi}_{s,t,n}\left(v_{t}\right)\right|\\[8pt]
&	=	&	\left|f(s)+\imath^{n}\imath^{-n}f(t)\right|\\[8pt]
&	=	&	f(s)+f(t),\\[8pt]
\end{array}\]
forcing equality.  Using the polarization identity,
\[
\left\langle v_{s},v_{t}\right\rangle_{R}=\frac{1}{4}\sum_{n=0}^{3}\imath^{n}\left\|v_{s}+\imath^{n}v_{t}\right\|_{R}
=\frac{1}{4}\sum_{n=0}^{3}\imath^{n}\left(f(s)+f(t)\right)
=0
\]
Thus, $\left(v_{s}\right)_{s\in S}$ is an orthogonal set in $R$.

Using Parseval's identity, for $s\neq t$,
\[\begin{array}{rcl}
\left\|v_{s}+v_{t}\right\|^{2}_{R}	&	=	&	\left\|v_{s}\right\|_{R}^{2}+\left\|v_{t}\right\|^{2}_{R}\\[8pt]
&	=	&	f(s)^{2}+f(t)^{2},\\[8pt]
\end{array}\]
but
\[\begin{array}{rcl}
\left\|v_{s}+v_{t}\right\|^{2}_{R}	&	=	&	\left(f(s)+f(t)\right)^{2}\\[8pt]
&	=	&	f(s)^{2}+2f(s)f(t)+f(t)^{2}.\\[8pt]
\end{array}\]
Together, these imply that if $s\neq t$, $f(s)f(t)=0$.  However, for distinct $s,t\not\in f^{-1}(0)$, this is impossible.  Therefore, this reflection can never have existed.

The case for $\mathbb{F}=\mathbb{R}$ follows by considering the real version of the polarization identity.

\end{proof}

From the proof, the issue here was due to the incompatibility of the universal property with Parseval's identity.  The universal property imposes that the norm on the reflection be an $\ell^{1}$-norm, like the case of $\mathbb{F}$-Banach spaces, but this cannot happen in a $\mathbb{F}$-Hilbert space other than $\mathbb{F}$ or $\mathbb{O}$.

\section{Universal Algebra for Normed Objects}

After performing all the constructions of Section \ref{construction}, consider the following diagram of categories and functors.
\[\xymatrix{
\mathbb{F}\mathbf{Alg}\ar@/^/[d]	&	\mathbb{F}\mathbf{BanAlg}_{\infty}\ar[d]\ar[l]	&	\mathbb{F}\mathbf{BanAlg}_{1}\ar@/^/[d]\ar[l]\\
\mathbb{F}\mathbf{Vec}\ar@/^/[d]\ar@/^/@{..>}[u]	&	\mathbb{F}\mathbf{Ban}_{\infty}\ar@/^/[d]\ar[l]	&	\mathbb{F}\mathbf{Ban}_{1}\ar@/^/[d]\ar[l]\ar@/^/@{..>}[u]\\
\Set\ar@/^/@{..>}[u]	&	\CSetI\ar[l]\ar@/^/@{..>}[u]	&	\CSetC\ar[l]\ar@/^/@{..>}[u]\\
}\]
Here, $\mathbb{F}\mathbf{Alg}$ and $\mathbb{F}\mathbf{Vec}$ denote the categories of $\mathbb{F}$-algebras with $\mathbb{F}$-algebra homomorphisms and $\mathbb{F}$-vector spaces with $\mathbb{F}$-linear maps, respectively.  The solid lines represent the natural forgetful functor, dropping the appropriate structure in each case.

The dotted arrows, however, represent the presence of a left adjoint functor.  In the purely algebraic cases, these are the free $\mathbb{F}$-vector space of a set and the tensor $\mathbb{F}$-algebra of an $\mathbb{F}$-vector space.  The normed cases are the constructions depicted in Section \ref{construction}.  Notice also that none of the horizontal functors have left adjoints since each lacks the necessary condition of preserving categorical products.  Likewise, none of them preserve categorical coproducts, so they cannot have right adjoints either.

Viewing these categories and construction functors together shows the parallels between the algebraic and the functional analytic theory.  For algebraic categories, $\Set$ plays a foundational role, allowing consideration of objects with minimal structure.  From the constructions of Section \ref{construction} and the properties found in Section \ref{chap2}, $\CSetC$ may play a similar role for normed objects, a category of objects with structure similar to normed objects, but minimal.

Plans are to investigate these relationships in subsequent papers, relying heavily on the foundation laid in this work.  In particular, the intuitive construction steps from Section \ref{construction}, building an algebraic object and competing in a universal norm, can be repeated for many well-known functional analytic objects:  operator spaces, operator algebras, C*-algebras, etc.  However, the failure result in Proposition \ref{hilbert} warns that care should be taken, as the universal property can interfere with the norm structure in adverse ways.

\bibliographystyle{../../bibstyles/hplain}
\bibliography{../../biblio.bib}

\begin{thebibliography}{10}

\bibitem{joyofcats}
Ji{\v{r}}{\'{\i}} Ad{{\'a}}mek, Horst Herrlich, and George~E. Strecker.
\newblock Abstract and concrete categories: the joy of cats.
\newblock {\em Repr. Theory Appl. Categ.}, (17):1--507, 2006.
\newblock Reprint of the 1990 original [Wiley, New York; MR1051419].

\bibitem{blackadar1985}
Bruce Blackadar.
\newblock Shape theory for {$C^\ast$}-algebras.
\newblock {\em Math. Scand.}, 56(2):249--275, 1985.

\bibitem{borceux1}
Francis Borceux.
\newblock {\em Handbook of categorical algebra. 1}, volume~50 of {\em
  Encyclopedia of Mathematics and its Applications}.
\newblock Cambridge University Press, Cambridge, 1994.
\newblock Basic category theory.

\bibitem{gerbracht}
Eberhard H.-A. Gerbracht.
\newblock {\em Elemente einer kombinatorischen Theorie der C*-Algebren:
  Pr{\"a}sentationen von C*-Algebren mittels Erzeugender und Relationen}.
\newblock PhD thesis, Technischen Universit{\"a}t Carolo-Wilhelmina zu
  Braunschweig, March 1998.

\bibitem{goodearl}
K.~R. Goodearl and P.~Menal.
\newblock Free and residually finite-dimensional {$C^*$}-algebras.
\newblock {\em J. Funct. Anal.}, 90(2):391--410, 1990.

\bibitem{grandis2004}
Marco Grandis.
\newblock Normed combinatorial homology and noncommutative tori.
\newblock {\em Theory Appl. Categ.}, 13:No. 7, 114--128, 2004.

\bibitem{grilliette0}
Will Grilliette.
\newblock {\em Formalizing Categorical and Algebraic Constructions in Operator
  Theory}.
\newblock PhD thesis, University of Nebraska -- Lincoln, March 2011.

\bibitem{gronbaek1992}
Niels Gr{\o}nb{\ae}k.
\newblock Weak and cyclic amenability for noncommutative {B}anach algebras.
\newblock {\em Proc. Edinburgh Math. Soc. (2)}, 35(2):315--328, 1992.

\bibitem{hadwin2003}
Don Hadwin, Llolsten Kaonga, and Ben Mathes.
\newblock Noncommutative continuous functions.
\newblock {\em J. Korean Math. Soc.}, 40(5):789--830, 2003.

\bibitem{lang}
Serge Lang.
\newblock {\em Algebra}, volume 211 of {\em Graduate Texts in Mathematics}.
\newblock Springer-Verlag, New York, third edition, 2002.

\bibitem{leptin1969}
H.~Leptin.
\newblock Die symmetrische {A}lgebra eines {B}anachschen {R}aumes.
\newblock {\em J. Reine Angew. Math.}, 239/240:163--168, 1969.

\bibitem{loring2009}
Terry~A. Loring.
\newblock From matrix to operator inequalities.
\newblock 02 2009, arXiv:0902.0102v1.

\bibitem{nassopoulos2008}
George~F. Nassopoulos.
\newblock A functorial approach to group {$C^*$}-algebras.
\newblock {\em Int. J. Contemp. Math. Sci.}, 3(25-28):1095--1102, 2008.

\bibitem{cat-top}
Joan~Wick Pelletier and Ji{\v{r}}{\'{\i}} Rosick{{\'y}}.
\newblock Generating the equational theory of {$C^*$}-algebras and related
  categories.
\newblock In {\em Categorical topology and its relation to analysis, algebra
  and combinatorics ({P}rague, 1988)}, pages 163--180. World Sci. Publ.,
  Teaneck, NJ, 1989.

\bibitem{pelletier}
Joan~Wick Pelletier and Ji{\v{r}}{\'{\i}} Rosick{{\'y}}.
\newblock On the equational theory of {$C^*$}-algebras.
\newblock {\em Algebra Universalis}, 30(2):275--284, 1993.

\end{thebibliography}

\end{document}